\documentclass[11pt,reqno, a4paper]{article}

\usepackage[LCY,T1]{fontenc}

\usepackage[english]{babel}
\usepackage[utf8]{inputenc}

\usepackage{algorithm, xpatch}
\usepackage[noend]{algpseudocode}
\usepackage{amsmath, amsthm, amssymb}
\usepackage{mathtools}
\usepackage{thmtools}
\usepackage{microtype}
\usepackage{caption}
\usepackage[a4paper, left=2.5cm, right=2.5cm, top=2.5cm, bottom=3cm, footskip=1.6cm]{geometry}
\numberwithin{equation}{section}

\DeclareMathOperator{\dist}{dist}

\newtheorem{theorem}{Theorem}[section]
\newtheorem{lemma}[theorem]{Lemma}
\newtheorem{corollary}[theorem]{Corollary}
\newtheorem{proposition}[theorem]{Proposition}
\newtheorem{remark}[theorem]{Remark}

\usepackage{hyperref}
\usepackage[nameinlink]{cleveref}
\crefname{prop}{proposition}{propositions}
\Crefname{prop}{Proposition}{Propositions}

\newcommand{\N}{\mathbb{N}}
\newcommand{\Z}{\mathbb{Z}}
\newcommand{\R}{\mathbb{R}}
\newcommand{\C}{\mathbb{C}}

\newcommand{\lin}{\ensuremath{\mathrm{lin}}}
\newcommand{\ip}[2]{\langle #1,#2\rangle}
\newcommand{\norm}[1]{\lVert #1\rVert}

\DeclareMathOperator{\spann}{span}

\DeclareMathOperator*{\argmax}{argmax}

\makeatletter
\def\blfootnote{\xdef\@thefnmark{}\@footnotetext}
\xpatchcmd{\algorithmic}{\itemsep\z@}{\itemsep=.4ex}{}{}
\makeatother

\makeatletter
\def\@seccntformat#1{\@ifundefined{#1@cntformat}%
   {\csname the#1\endcsname\quad}  
   {\csname #1@cntformat\endcsname}
}

\let\oldappendix\appendix 
\renewcommand\appendix{%
    \oldappendix
    \newcommand{\section@cntformat}{\appendixname~\thesection\quad}
    \renewcommand{\theHsection}{appendix.\Alph{section}}
    \renewcommand{\theHsubsection}{appendix.\Alph{section}.\arabic{subsection}}
    \renewcommand{\theHsubsubsection}{appendix.\Alph{section}.\arabic{subsection}.\arabic{subsubsection}}
    \renewcommand{\theHequation}{appendix.\Alph{section}.\arabic{equation}}
    \renewcommand{\theHtheorem}{appendix.\Alph{section}.\arabic{theorem}}
    \renewcommand{\theHfigure}{appendix.\Alph{section}.\arabic{figure}}
    \renewcommand{\theHtable}{appendix.\Alph{section}.\arabic{table}}
}
\makeatother
\usepackage{todonotes}
\usepackage{xcolor}

\begin{document}

\title{Greedy sampling designs via reduced basis methods:\\optimal recovery in the uniform norm}

\author{Sebastian Neumayer $\!{}^{a}$,
Kateryna Pozharska $\!{}^{a,b}$,
Tino Ullrich $\!{}^{a,}$\footnote{Corresponding author, Email:
tino.ullrich@math.tu-chemnitz.de}\\\\
${}^{a}\!$ Chemnitz University of Technology, Faculty of Mathematics\\[2mm]
${}^{b}\!$ Institute of Mathematics of NAS of Ukraine}

\date{\today}

\maketitle

\begin{abstract}
We study optimal sampling recovery in reproducing kernel Hilbert spaces (RKHS) in the uniform norm.
For every RKHS with bounded kernel, we establish new comparisons between linear sampling widths and Gelfand widths that overcome the known square-root gap, without requiring a measure or a Christoffel-type condition.
Our bounds rely on nested sampling designs obtained by kernel interpolation at (weak) \(P\)-greedy points. Under additional (polynomial) decay assumptions the decay rate of the Gelfand widths directly transfers to the sampling widths. With either a  logarithmic oversampling or passing to the square root of the Gelfand widths we obtain a direct comparison (requiring no decay assumption) between them. This is particularly effective for super-polynomial decay, such as in Paley–Wiener spaces.
Our results follow from representations of both widths in terms of kernel translates and yield, in the opposite direction, a new existence result for a sharp reduced basis selection.
Numerical experiments for Legendre, mixed-Sobolev, and Paley–Wiener kernels illustrate our findings.
\end{abstract}

\section{Introduction}\label{sec:intro}
Many problems in the applied sciences require reconstructing a quantity of interest from a finite number of measurements.
Mathematically, we model this quantity by a function $f\colon D\to\C$ on a (possibly high-dimensional) set $D$, and the corresponding measurements can be, for example, Fourier coefficients, local averages, or point evaluations $f(x_1),\ldots,f(x_m)$.
The task is to accurately recover $f$ from this information.
For this, one commonly assumes that $f \in \cal F$ belongs to a (regularity) class, where $\cal F$ usually denotes a subset of a normed function space $F$ continuously embedded into $B(D)$, the space of bounded complex-valued functions on $D$ equipped with $\norm{f}_\infty \coloneqq \sup_{x\in D}|f(x)|$.
A key question is how the achievable worst-case accuracy depends on the type of measurements.
In this paper, we focus on the recovery problem in the uniform norm and compare two extremes: arbitrary linear information, leading to the \emph{Gelfand widths}, and sample information, leading to the \emph{linear sampling widths}, which are defined, respectively, as
\begin{align}
c_n(\mathcal{F})_\infty &\coloneqq \inf_{\substack{\psi\colon\C^n\to B(D)\\ N\colon F \to \C^n \text{ bounded linear}}}
\sup_{f\in \mathcal{F}}\norm{f-\psi(Nf)}_\infty,\\
g_m^{\lin}(\mathcal{F})_\infty &\coloneqq \inf_{\substack{x_1,\dots,x_m\in D\\ \varphi\colon \C^m\to B(D) \text{ linear}}}
\sup_{f\in \mathcal{F}} \bigl\| f-\varphi\bigl(f(x_1),\ldots,f(x_m)\bigr)\bigr\|_\infty \label{eq:glin}
\end{align}
with the convention $g_0^{\lin}(\mathcal{F})_\infty = c_0(\mathcal{F})_\infty =\sup_{f\in \mathcal{F}}\norm{f}_\infty$.
Thus, $g_m^{\lin}(\mathcal{F})_\infty$ is the minimal worst-case error of a linear algorithm using $m$ samples, while $c_n(\mathcal{F})_\infty$ allows arbitrary algorithms and arbitrary linear information.
Consequently, the Gelfand widths $c_m(\mathcal{F})_\infty$ are a lower bound for $g_m^{\lin}(\mathcal{F})_\infty$ and serve as a benchmark for the optimality of recovery methods.
The interesting direction is the converse one: how much is lost by restricting oneself to sample information?
The recovery problem in the uniform norm has attracted considerable attention in recent years \cite{PU}, \cite{KPUU23, KPUU24}, following a breakthrough in $L_2$-recovery \cite{NaSchul22, Te21, DoKrUl23}.

Our goal in this paper is to establish a new transference principle which allows for using greedy methods from the reduced basis community. Therefore, our study involves several additional notions of approximation errors.
We write $g_m(\mathcal{F})_\infty$ for the quantity defined as in \eqref{eq:glin} but with arbitrary, possibly non-linear recovery maps $\varphi\colon\C^m\to B(D)$.
For the classes considered here, namely an RKHS with bounded kernel, we have $g_m=g_m^{\lin}$, see Proposition~\ref{prop:gm}.
Relations like this are classical and known to hold in wider situations, see \cite[Sect.\ 4.2.2]{NW1}.
Another relevant quantity is the \emph{Kolmogorov $n$-width}, which measures approximation by linear subspaces.
For a subset $\mathcal F$ of a normed space $X$, it is defined by
\begin{equation}\label{eq:dn-general}
  d_n(\mathcal F)_X \coloneqq \inf_{\substack{V\subset X\\ \dim V\le n}} \ \sup_{a\in\mathcal F} \ \dist_X(a,V) .
\end{equation}
It is important to note that $\mathcal F$ does not need to have any structure here (convexity, symmetry, unit ball...). In addition, $V \subset X$ need not be spanned by elements of $\mathcal F$. 
The latter restriction leads to the restricted version $\tau_n(\mathcal{F})_X$ of the Kolmogorov width defined in \eqref{eq:tau_m}, see also \cite{Wojtaszczyk15}.   
 
Various relations between sampling, Gelfand and Kolmogorov widths have been investigated in recent papers, see for instance \cite{KPUU24}; for the current state of the art we refer to the survey \cite{Survey26} and the references therein.
According to \cite[Thm.~20 and eq.~(14)]{KPUU24}, whenever $\mathcal{F}$ is the unit ball of a normed function space that embeds continuously into $B(D)$, it holds that
\begin{equation}\label{eq:KPUU-sqrtn}
g_{2n}^{\lin}(\mathcal{F})_\infty \le C\sqrt{n}\, d_n(\mathcal{F})_\infty \leq C\sqrt{n}\,c_n(\mathcal{F})_\infty,
\end{equation}
where $C$ is an absolute constant.
In general, the polynomial order of the factor on the right-hand side cannot be removed, even if polynomial oversampling is allowed, in the sense that for no fixed $\kappa$ does $g_{n^{\kappa}}^{\lin} (\mathcal{F})_\infty  \lesssim c_n (\mathcal{F})_\infty$ hold uniformly for $\mathcal{F}$. 
A counterexample is given by the unit balls of the Besov spaces $B^s_{1,q}([0,1])$ with $s>1$ and $1\leq q \leq \infty$, see \cite[Sec.~5.3]{KPUU24}.
 
Throughout this paper, we restrict ourselves to $\mathcal F = \mathcal H_K$ (slight abuse of notation) being a reproducing kernel Hilbert space (RKHS) on $D$ with bounded kernel $K\colon D\times D\rightarrow\mathbb{C}$ of radius
\begin{equation}\label{eq:R}
R\coloneqq\sup_{x\in D}\sqrt{K(x,x)}<\infty ,
\end{equation}
and let $B_{\mathcal H_K}$ denote its unit ball.
By \eqref{eq:R} and the reproducing property $f(x)=\ip{f}{K(\cdot,x)}_{\mathcal H_K}$, which holds for all $f\in \mathcal H_K$, $x\in D$, one has $g_0^{\lin}(B_{\mathcal H_K})_\infty=c_0(B_{\mathcal H_K})_\infty=R$.
In this setting, a constant-factor comparison is known under additional structural assumptions.
Specifically, \cite[Thm.~11]{KPUU23} assumes a finite measure $\mu$ on $D$ such that every $f\in\mathcal H_K$ is $\mu$-measurable and $\mathcal H_K\subset L_2(\mu)$, and a singular system yielding the Mercer representation $K(x,y)=\sum_{k\ge1}\sigma_k^2b_k(x)\overline{b_k(y)}$.
If the Christoffel-type condition
\begin{equation}\label{eq:cond8}
\sup_{n\in\N}\sup_{x\in D}
\frac1n\sum_{k\le n}|b_k(x)|^2\le B
\end{equation}
is satisfied, then
\begin{equation}\label{eq:thm11}
g_{\kappa n}^{\lin}(B_{\mathcal H_K})_\infty
\le \sqrt{BC'\mu(D)}\,
c_n(B_{\mathcal H_K})_\infty
\end{equation}
for an absolute constant $\kappa>1$.
This leads us to the following two questions:

\begin{enumerate}
\item[(Q1)] Do the usual scales of decay of $c_n(B_{\mathcal H_K})_\infty$ transfer to $g_n^{\lin}(B_{\mathcal H_K})_\infty$ for every RKHS with bounded kernel?
\item[(Q2)] Can the factor $\sqrt{n}$ in the $c_n$ bound of \eqref{eq:KPUU-sqrtn} be removed for $\mathcal{F}=B_{\mathcal H_K}$, i.e.\ when $\mathcal{F}$ is the unit ball of an RKHS with $R<\infty$?
\end{enumerate}
To answer these questions, our main tool is  the ``dictionary'' of kernel translates
\begin{equation}\label{eq:dictionary}
\mathcal{K}\coloneqq\{K(\cdot,x)\colon x\in D\}\subset \mathcal H_K .
\end{equation}
The key observation for our analysis is that $g_m^{\lin}(B_{\mathcal H_K})_\infty$ is the smallest supremum of the kernel power function attainable by a choice of $m$ sampling points, i.e.,
\begin{equation}
    g_m(B_{\mathcal H_K})_\infty = g_m^{\lin}(B_{\mathcal H_K})_\infty = \tau_m(\mathcal{K})_{\mathcal{H}_K} = \inf_{|P|\le m} \sup_{x\in D}\dist_{\mathcal H_K}\bigl(K(\cdot,x), V_P\bigr),
\end{equation}
see Proposition \ref{prop:gm} for the notation.
We further have 
\begin{equation}\label{eq:cn-dn}
c_n(B_{\mathcal H_K})_\infty = d_n(\mathcal{K})_{\mathcal H_K}
= \inf_{\substack{V\subset\mathcal H_K\\ \dim V \leq n}} \sup_{x\in D} \dist_{\mathcal H_K}\bigl(K(\cdot,x),V\bigr),
\end{equation}
where $d_n(\mathcal{K})_{\mathcal H_K}$ denotes the Kolmogorov $n$-width \eqref{eq:dn-general} of $\mathcal{K}$ for the target space $X=\mathcal H_K$, i.e.\ \emph{measured in the norm of $\mathcal H_K$} and with subspaces taken from $\mathcal H_K$, see Propositions~\ref{prop:gm} and~\ref{prop:cn}.

What remains is a geometric question: How well can $\mathcal{K}$ be approximated by spans of finitely many of
\emph{its own} elements, compared with arbitrary subspaces of the same dimension?
This is an instance of reduced basis subset selection, so that the theory of greedy algorithms in Hilbert spaces \cite{BCDDPW11,DPW13} applies.
Through the kernel-translate dictionary, comparisons between greedy and Kolmogorov widths become comparisons between sampling and Gelfand widths.
The Carl-type comparison in Theorem~\ref{thm:rates} answers (Q1) for the usual decay scales.
Theorem~\ref{thm:log} further yields an absolute constant at the price of oversampling, namely a partial answer to (Q2).
Further, reversing the transfer gives a sharp existence result for reduced basis subset selection.
Through the kernel-translate dictionary, the weak greedy method constructs a nested sampling design by selecting $\gamma$-approximate maximizers of the power function.
Kernel interpolation at the resulting points then provides the recovery, see Remark~\ref{rem:nested}.
 
\section{Main results}\label{sec:MainResults}
First, we relate linear sampling widths and Gelfand widths in the uniform norm.
\begin{theorem}[Carl-type comparison]\label{thm:rates}
Let $L\colon[0,\infty)\to(0,\infty)$ be non-decreasing with $L(0)>0$ and
\begin{equation}\label{eq:LambdaL}
  \Lambda_L\coloneqq\sup_{s\ge1}\frac{L(8s)^2}{L(2s)L(s)}<\infty .
\end{equation}
Then for every set $D$, every RKHS $\mathcal H_K$ on $D$ with bounded kernel $K$, and every $n\in\N$, we have
\begin{equation}\label{eq:carl}
    \sup_{0\le k\le n} L(k) g_k^{\lin}(B_{\mathcal H_K})_\infty 
    \le C_L\sup_{0\le k\le n} L(k) c_k(B_{\mathcal H_K})_\infty,
    \qquad C_L\coloneqq\max\bigl\{L(3)/L(0), 8\Lambda_L\bigr\}.
  \end{equation}
\end{theorem}

\begin{remark}[Regularly varying functions]
\label{rem:RV}
Condition \eqref{eq:LambdaL} is satisfied by every non-decreasing $L$ that is regularly varying in the sense of Karamata \cite{BGT87} of index $\alpha\ge0$, i.e.\ $\lim_{u\to\infty}L(\lambda u)/L(u)=\lambda^{\alpha}$ for all $\lambda>0$.
Indeed, $L(8s)^2/(L(2s)L(s))\to 2^{5\alpha}$ by the defining limit at $\lambda=8$ and $\lambda=2$. Standard monotonicity arguments bound the quotient for bounded $s$.
Important examples are $L(k)=(1+k)^{\alpha}$ and $L(k)=(1+k)^{\alpha}(\log(2+k))^{\beta}$ with $\alpha>0$, $\beta\in\R$. For $L(k)=(1+k)^{\alpha}$, one has $\Lambda_L=2^{5\alpha}$, hence $C_L\le 2^{5\alpha+3}$. 
\end{remark}
The term  ``Carl-type comparison'' originates in Carl's famous inequality $\sup_{1\le k\le n}k^{\rho}e_k(T)\le c_\rho\sup_{1\le k\le n}k^{\rho}a_k(T)$ between entropy numbers $e_k$ and approximation numbers $a_k$, see \cite{Carl81,CarlStephani90}.
Specializing $L$ gives the following.

\begin{corollary}[Polynomial rate transfer]\label{cor:rates}
    Let $\mathcal H_K$ be an RKHS on $D$ with bounded kernel~$K$.
    \begin{itemize}
    \item[(i)]
    If for all $n\in\N$ it holds $c_n(B_{\mathcal H_K})_\infty  \le C_0(1+n)^{-\alpha}$ with some $\alpha, C_0 >0$, then it holds
    \begin{equation}
        g_{n}^{\lin}(B_{\mathcal H_K})_\infty  \le 2^{5\alpha+3}\max\{C_0,R\} (1+n)^{-\alpha}.
    \end{equation}

    \item[(ii)] If $c_n(B_{\mathcal H_K})_\infty \le C_0 n^{-\alpha}(\log n)^{-\beta}$ for all $n\ge2$, with either $\alpha>0$, $\beta\in\R$, or $\alpha=0$, $\beta>0$, then
    \begin{equation}
        g_n^{\lin}(B_{\mathcal H_K})_\infty \le C(\alpha,\beta)\max\{C_0,R\}
        n^{-\alpha}(\log n)^{-\beta}, \qquad n\geq 2.
    \end{equation}
\end{itemize}
\end{corollary}
In part (ii), for $\alpha>0$, $\beta<0$, the weight $L(k)=(1+k)^{\alpha}(\log(2+k))^{\beta}$ is not non-decreasing.
Theorem~\ref{thm:rates} is then applied to the majorant $\widetilde L(k)\coloneqq\max_{0\le j\le k}L(j)$, which leaves the rate unchanged.
Part~(i) recovers the constant obtained in
\cite[Cor.~3.3(ii)]{DPW13} for $\gamma=\tfrac12$; part~(ii), and more generally the comparison of Theorem~\ref{thm:rates}, does not follow from it.
Our result also does not contradict the sharpness of \eqref{eq:KPUU-sqrtn} established in \cite[Sec.~5.3]{KPUU24}:
the corresponding Besov ball is non-Hilbert, so Proposition~\ref{prop:cn} does not apply.

In summary, Theorem~\ref{thm:rates} replaces the factor $\sqrt n$ in \eqref{eq:KPUU-sqrtn} by a rate-dependent constant together with the weighted suprema in \eqref{eq:carl}.
By contrast, condition \eqref{eq:cond8} yields an absolute constant independent of the decay rate and remains valid without any decay assumption.
The decay assumption can also be dropped in our setting, at the cost of taking a square root.
 
\begin{theorem}[Square root bound]\label{thm:geom}
Let $\mathcal H_K  $ be an RKHS on $D$ with bounded kernel $K$ and $R = \sup_{x\in D}\sqrt{K(x,x)}$. Then for all $n\in\N$ it holds
\begin{equation}
g_{2n}^{\lin}(B_{\mathcal H_K})_\infty  \le \sqrt{8R \, c_n(B_{\mathcal H_K})_\infty}.
\end{equation}
\end{theorem}

Theorem~\ref{thm:geom} transfers the greedy comparison
\cite[Cor.~3.3(i)]{DPW13} through \eqref{eq:greedy-vs-g}.
Its exponential-rate consequence below corresponds to \cite[Cor.~3.3(iii)]{DPW13}.
This requires no decay assumption and improves \eqref{eq:KPUU-sqrtn} whenever $c_n(B_{\mathcal H_K})_\infty\ge 8R/(C^2n)$.
For example, if $c_n(B_{\mathcal H_K})_\infty\le C_0\exp(-\theta n^\alpha)$, which lies outside Karamata's class, then with $m=2n$ it holds
\begin{equation}
g_m^{\lin}(B_{\mathcal H_K})_\infty
\le \sqrt{8RC_0}\,
\exp\bigl(-\theta m^\alpha/2^{1+\alpha}\bigr).
\end{equation}
Thus, the order $m^\alpha$ in the exponent is preserved, although its
constant is reduced by $2^{1+\alpha}$. 
The estimate \eqref{eq:KPUU-sqrtn} gives a better exponential constant up to a
polynomial prefactor, but is nonconstructive, whereas
Theorem~\ref{thm:geom} is constructive and requires no prescribed decay.
The Paley--Wiener example in Subsection~\ref{sec:Numerical_results}
exhibits such a decay.

An absolute constant is known under exponential oversampling:
if $D\subset\R^d$ and $\mathcal F\subset C(D)$ are compact, then $g_{9^n}(\mathcal F)_\infty\le5d_n(\mathcal F)_\infty$
\cite{KKT21}.
For convex and symmetric $\mathcal F$, one moreover has $d_n(\mathcal F)_\infty\le c_n(\mathcal F)_\infty$.
The entropy estimate $g_n^{\lin}(\mathcal F)_\infty\le(n+1)\varepsilon_n(\mathcal F)_\infty$
from \cite{U26} avoids oversampling, but pays a factor of order $n$.
For bounded-kernel RKHSs, the following result reduces the exponential sampling budget to an oversampling factor logarithmic in $R/c_n$.
 
\begin{theorem}[Direct comparison via oversampling]\label{thm:log}
Let $\mathcal H_K$ be an RKHS on $D$ with bounded kernel $K$ with $R\coloneqq\sup_{x\in D}\sqrt{K(x,x)}$ and let $n\in\N$ be such that $c_n\coloneqq c_n(B_{\mathcal H_K})_\infty>0$.
Put 
\begin{equation}
    m\coloneqq \lceil 4n\log_2(2+R/c_n) \rceil.
\end{equation} 
Then, it holds
\begin{equation}
g_{m}^{\lin}(B_{\mathcal H_K})_\infty  \le 16 c_n(B_{\mathcal H_K})_\infty.
\end{equation}
\end{theorem}

Theorem~\ref{thm:rates} states that under additional (polynomial) decay assumptions the decay rate of the Gelfand widths directly transfers to the sampling widths. On the other hand, Theorem~\ref{thm:log} gives a constructive direct comparison under decay-dependent oversampling.
If $c_n\asymp n^{-\alpha}$, the latter uses $m\asymp n\log n$, while Corollary~\ref{cor:rates} avoids oversampling at the price of a rate-dependent constant.
If $c_n\asymp\exp(-\theta n^\alpha)$, it uses $m\asymp n^{1+\alpha}$, and Theorem~\ref{thm:geom} provides a constructive alternative.

Finally, we reverse the transference principle.
Our subset width $\tau_m(\mathcal A)_X$ is denoted by $\bar d_m$ in \cite{Wojtaszczyk15} and called the greedy Kolmogorov width in \cite[Sec.~1.3]{DPSW26}.
At the same index, \cite[Thm.~3.1]{Wojtaszczyk15} gives
$\tau_m(\mathcal A)_X\le(m+1)d_m(\mathcal A)_X$ for arbitrary compact subsets of Banach spaces and shows that the order $m$ is necessary; the Hilbert-space case goes back to \cite[Thm.~4.1]{BCDDPW11}. 
Unlike the weak-greedy error $\sigma_m$, the quantity $\tau_m$ minimizes over all subsets of at most $m$ elements and satisfies $\tau_m\le\sigma_m$.
Hence, a bound for $\tau_m$ does not imply the analogous estimate for weak $P$-greedy sampling.
Earlier comparisons for $\sigma_m$ carried a factor $m2^m$ \cite{BuMaPaPrTu12,DPW13}, later improved to the sharp factor $2^m$ in \cite[(4.4)]{BCDDPW11}.
Reversing the dictionary argument and applying \eqref{eq:KPUU-sqrtn} gives the following $\sqrt m$ bound under oversampling, which is optimal for general normed spaces even under constant oversampling.

\begin{theorem}[Optimal reduced basis selection]\label{cor:best-subset}\label{thm:basis_sel}
Let $\mathcal A$ be a compact subset of a normed space $X$ and define
\begin{equation}\label{eq:tau_m}
    \tau_m(\mathcal A)_X\coloneqq \inf_{a_1,\ldots,a_m\in\mathcal A} \sup_{a\in\mathcal A} \dist_X\bigl(a,\spann\{a_1,\ldots,a_m\}\bigr).
\end{equation}
Then, with the absolute constant $C$ from \eqref{eq:KPUU-sqrtn}, it holds that
\begin{equation}\label{eq:best-subset}
\tau_{2m}(\mathcal A)_X\le C\sqrt m\,d_m(\mathcal A)_X,
\qquad m\in\N.
\end{equation}
\end{theorem}

\section{The kernel-translates dictionary}\label{sec:dict}
 
Recall from Section~\ref{sec:intro} that $\mathcal H_K$ is an RKHS on $D$ with bounded kernel $K$, i.e.\ $R<\infty$ in \eqref{eq:R}.
Let
\begin{equation}
    \mathcal{K} \coloneqq \{a_x : x\in D\}\subset \mathcal H_K, \qquad a_x\coloneqq K(\cdot,x), \qquad \norm{a_x}_{\mathcal H_K}=\sqrt{K(x,x)}\le R.
\end{equation}
For a finite set $P\subset D$, let $V_P\coloneqq\spann\{a_x : x\in P\}$ and define the \emph{power function}
\begin{equation}
    \mathrm{Pow}_P(x) \coloneqq \dist_{\mathcal H_K}\bigl(a_x, V_P\bigr) = \|(I-\Pi_{V_P})a_x\|_{\mathcal H_K},
\end{equation}
where $\Pi_{V_P}$ denotes the $\mathcal H_K$-orthogonal projection onto $V_P$.
This leads to
\begin{equation}\label{power_f}
    \tau_m(\mathcal{K})_{\mathcal{H}_K} = \inf_{|P|\le m} \sup_{x\in D} \mathrm{Pow}_P(x) = \inf_{|P|\le m} \|\mathrm{Pow}_P\|_\infty
\end{equation}
and $\tau_0(\mathcal{K})=d_0(\mathcal{K})_{\mathcal H_K}=\sup_{x\in D}\norm{a_x}_{\mathcal H_K}=R$.
The following observation identifies the orthogonal projection with a sampling algorithm.

\begin{lemma}[Kernel interpolation is a sampling algorithm]\label{lem:interp}
Let $P=\{x_1,\dots,x_k\}\subset D$ be finite, let $\mathbf K[P]\coloneqq(K(x_i,x_j))_{i,j\le k}$ be the associated Gram matrix with Moore--Penrose inverse $\mathbf K[P]^{+}$, and let $\mathbf f[P]\coloneqq(f(x_j))_{j\le k}$ collect the samples of $f\in\mathcal H_K$.
Then, the $\mathcal H_K$-orthogonal projection onto $V_P \coloneqq \spann\{K(\cdot,x_1),\ldots, K(\cdot,x_k)\}$ is given by
\begin{equation}\label{eq:interp}
    \Pi_{V_P}f=\sum_{i\le k}\bigl(\mathbf K[P]^{+}\mathbf f[P]\bigr)_i K(\cdot,x_i), \qquad f\in\mathcal H_K .
\end{equation}
Thus, $\Pi_{V_P}f$ is linear in the samples, interpolates them, and has minimal $\mathcal H_K$-norm among all functions with these values.
\end{lemma}

\begin{proof}
We write $\Pi_{V_P}f=\sum_{i\le k}\alpha_i a_{x_i}$ with coefficient vector $\boldsymbol{\alpha}=(\alpha_i)_{i\le k}$.
By the reproducing property, the normal equations $\ip{f-\Pi_{V_P}f}{a_{x_j}}_{\mathcal H_K}=0$ read $\mathbf K[P]\boldsymbol{\alpha}=\mathbf f[P]$.
This consistent system has the solution $\boldsymbol{\alpha}=\mathbf K[P]^{+}\mathbf f[P]$, which proves \eqref{eq:interp}, linearity in the samples, and interpolation.
Finally, if $g\in\mathcal H_K$ satisfies $g(x_j)=f(x_j)$ for all $j\le k$, then $g-\Pi_{V_P}f\perp V_P$ and hence $\norm{g}_{\mathcal H_K}^2=\norm{\Pi_{V_P}f}_{\mathcal H_K}^2+\norm{g-\Pi_{V_P}f}_{\mathcal H_K}^2\ge\norm{\Pi_{V_P}f}_{\mathcal H_K}^2$.
\end{proof}
 
\begin{proposition}\label{prop:gm}
For every $m\in\N_0$, it holds that
\begin{equation}\label{eq:chain_prop_gm}
    g_m(B_{\mathcal H_K})_\infty = g_m^{\lin}(B_{\mathcal H_K})_\infty = \tau_m(\mathcal{K})_{\mathcal{H}_K}.
\end{equation}
\end{proposition}
 
\begin{proof}
Since linear algorithms are a subclass of all algorithms, we have
\begin{equation}\label{eq:gm-triv}
    g_m(B_{\mathcal H_K})_\infty \le g_m^{\lin}(B_{\mathcal H_K})_\infty.
\end{equation}
Below, we prove $g_m^{\lin}(B_{\mathcal H_K})_\infty\le\tau_m(\mathcal{K})_{\mathcal{H}_K}$ and $g_m(B_{\mathcal H_K})_\infty\ge\tau_m(\mathcal{K})_{\mathcal{H}_K}$.
Together with \eqref{eq:gm-triv}, this implies
\begin{equation}
    \tau_m(\mathcal{K})_{\mathcal{H}_K} \le g_m(B_{\mathcal H_K})_\infty \le g_m^{\lin}(B_{\mathcal H_K})_\infty \le \tau_m(\mathcal{K})_{\mathcal{H}_K},
\end{equation}
so that both equalities in \eqref{eq:chain_prop_gm} hold.
 
First, we establish the upper bound for $g_m^{\lin}$.
Fix a set $P=\{x_1,\dots,x_k\}\subset D$ with $k=|P|\le m$.
This is admissible in \eqref{eq:glin}, since repeating a node changes neither $V_P$ nor the resulting algorithm, and by Lemma~\ref{lem:interp} kernel interpolation $f\mapsto\Pi_{V_P}f$ is a linear algorithm using only the samples $\mathbf f[P]$.
Since $f-\Pi_{V_P}f\perp V_P$, we get for $f\in B_{\mathcal H_K}$ and $x\in D$ that
\begin{align}
    |f(x)-\Pi_{V_P}f(x)| &= \bigl|\ip{f-\Pi_{V_P}f}{a_x}_{\mathcal H_K}\bigr|
    = \bigl|\ip{f-\Pi_{V_P}f}{(I-\Pi_{V_P})a_x}_{\mathcal H_K}\bigr| \notag\\
    & \le \norm{f-\Pi_{V_P}f}_{\mathcal H_K} \norm{ (I-\Pi_{V_P})a_x  }_{\mathcal H_K}\le \norm{f}_{\mathcal H_K} \mathrm{Pow}_P(x) \le \mathrm{Pow}_P(x).
\end{align}
Taking the supremum over $x\in D$ and over $f\in B_{\mathcal H_K}$, it holds for every $P$ with $|P|\le m$ that
\begin{equation}
    g_m^{\lin}(B_{\mathcal H_K})_\infty \le
    \sup_{f\in B_{\mathcal H_K}}\norm{f-\Pi_{V_P}f}_\infty \le \norm{\mathrm{Pow}_P}_\infty.
\end{equation}
 
Now, we show the lower bound for $g_m$.
To this end, we use a fooling argument, constructing two admissible functions that produce the same data but are far apart, so that no algorithm can reconstruct both equally well.
Let an arbitrary algorithm based on the samples $f(x_1),\dots,f(x_k)$ at locations $P=\{x_1,\dots,x_k\}$ with $k=|P|\le m$ be given, and let $A\in B(D)$ be its output on the data $(0,\dots,0)$.
We fix $0<\delta<\norm{\mathrm{Pow}_P}_\infty$ and pick $x^*$ with $\mathrm{Pow}_P(x^*)\ge \norm{\mathrm{Pow}_P}_\infty-\delta>0$.
If $\norm{\mathrm{Pow}_P}_\infty=0$ there is nothing to prove.
The function
\begin{equation}
    h \coloneqq \frac{(I-\Pi_{V_P})a_{x^*}}{\mathrm{Pow}_P(x^*)}
\end{equation}
satisfies $\norm{h}_{\mathcal H_K}=1$, $h(x_i)=\ip{h}{a_{x_i}}_{\mathcal H_K}=0$ for all $i\le k$, and $h(x^*)=\ip{h}{a_{x^*}}_{\mathcal H_K}=\mathrm{Pow}_P(x^*)$.
Thus, both $h$ and $-h$ belong to $B_{\mathcal H_K}$ and produce the same data $h(x_i)=-h(x_i)=0$ for $x_i\in P$, so that the
algorithm returns the same $A$ in either case. By the triangle inequality, we have that
\begin{equation}
    \max\bigl\{\norm{h-A}_\infty, \norm{-h-A}_\infty\bigr\} \ge \norm{h}_\infty \ge |h(x^*)| = \mathrm{Pow}_P(x^*)\ge\norm{\mathrm{Pow}_P}_\infty-\delta,
\end{equation}
so that the worst-case error of the recovery algorithm is at least $\norm{\mathrm{Pow}_P}_\infty-\delta$.
Letting $\delta\to0$ and taking the infimum over $P$ gives $g_m(B_{\mathcal H_K})_\infty\ge\tau_m(\mathcal{K})_{\mathcal{H}_K}$.
\end{proof}

\begin{proposition}\label{prop:cn}
For every $n\in\N_0$, it holds that $c_n(B_{\mathcal H_K})_\infty  = d_n(\mathcal{K})_{\mathcal H_K}$.
\end{proposition}

\begin{proof}
We use the relation $c_n(B_{\mathcal H_K})_\infty= a_n(B_{\mathcal H_K})_\infty$ in the case $F = \mathcal{H}_{K}$, where 
\begin{equation}
a_n(B_F)_X \coloneqq \inf_{\substack{T\colon F\to X\\ \rm{rank } T\le n}}\sup_{f\in B_F}\norm{f-Tf}_X,
\end{equation}
and $T$ ranges through all bounded linear operators $T\colon F\to X$ of rank at most $n$, and $X$ is a normed space such that $F\hookrightarrow X$. This identity is well-known, see \cite[Sect.\ 4.2.2]{NW1}.

Let $T$ be a linear operator of rank at most $n$.
We may assume $T = \sum_{j\le n}\ip{\cdot}{v_j}_{\mathcal H_K} g_j$ with $v_j\in \mathcal H_K$ and $g_j\in B(D)$, the functionals $f\mapsto\ip{f}{v_j}_{\mathcal H_K}$ being the Riesz representations of the bounded coefficient functionals of $T$.
Then for $f\in B_{\mathcal H_K}$ and $x\in D$, it holds $(f-Tf)(x) = \langle f, a_x-\sum_{j\le n}\overline{g_j(x)} v_j\rangle_{\mathcal H_K}$, so that
\begin{equation}
    \sup_{f\in B_{\mathcal H_K}}|(f-Tf)(x)| = \Bigl\| a_x-\sum_{j\le n}\overline{g_j(x)} v_j\Bigr\|_{\mathcal H_K}.
\end{equation}
Taking the supremum over $x$ shows that every linear operator $T$ of rank at most $n$ satisfies $\sup_{f\in B_{\mathcal H_K}}\norm{f-Tf}_\infty\ge\sup_{x\in D}\dist_{\mathcal H_K}(a_x,W)$ with $W=\spann\{v_1,\dots,v_n\}$, which in turn implies $a_n(B_{\mathcal H_K})_\infty\ge d_n(\mathcal{K})_{\mathcal H_K}$.

Conversely, fix $W\subset \mathcal H_K$ with $\dim W=l\le n$ and an orthonormal basis $w_1,\dots,w_l$ of $W$.
Then, let $T\coloneqq\sum_{j\le l}\ip{\cdot}{w_j}_{\mathcal H_K} w_j$, so that $Tf$ is the orthogonal projection of $f$ onto $W$ and $T$ has rank at most $n$. 
Its coefficient functionals are $x\mapsto w_j(x)$, which by the reproducing property are bounded by $R$ and hence admissible elements of $B(D)$. 
This gives $a_n(B_{\mathcal H_K})_\infty\le \sup_{x\in D}\dist_{\mathcal H_K}(a_x,W)$ and therefore $a_n(B_{\mathcal H_K})_\infty=d_n(\mathcal{K})_{\mathcal H_K}$.
\end{proof}

Propositions~\ref{prop:gm} and~\ref{prop:cn} form the transference principle. 
Proposition~\ref{prop:gm} is folklore in kernel approximation; see \cite{SH17} and the references therein, while Proposition~\ref{prop:cn} is implicit in optimal-recovery theory \cite{NW3,OP95,KWW08}.
The combination of \eqref{eq:greedy-vs-g} with the greedy estimates of \cite{DPW13} was already used for kernels in \cite{SH17}, but without identifying $d_n(\mathcal K)_{\mathcal H_K}$ with $c_n(B_{\mathcal H_K})_\infty$.

\section{Greedy sampling designs, Gelfand widths and reduced bases}\label{sec:greedy}

All results of this section are obtained from the \emph{weak greedy algorithm} of \cite{BCDDPW11,DPW13}.
Let $\mathcal H$ be a Hilbert space, let $\mathcal{F}\subset B_{\mathcal H}$ be a subset of its unit ball, and fix a parameter $\gamma\in(0,1]$.
Starting from $V_0\coloneqq\{0\}$, the algorithm selects $f_0,f_1,\dots\in\mathcal{F}$ recursively: given $V_j\coloneqq\spann\{f_0,\dots,f_{j-1}\}$, it picks any $f_j\in\mathcal F$ with
\begin{equation}\label{eq:weak-greedy}
    \dist(f_j,V_j)\ \ge\ \gamma \sigma_j,
    \qquad
    \sigma_j\coloneqq\sup_{f\in\mathcal{F}}\dist(f,V_j).
\end{equation}
Since $V_j\subset V_{j+1}$, the sequence $(\sigma_j)_{j\ge0}$ is non-increasing.
In particular, $\sigma_0=\sup_{f\in\mathcal F}\norm{f}_{\mathcal H}\le1$.
For $\gamma<1$, such a $f_j$ exists for every $\mathcal F$, since \eqref{eq:weak-greedy} asks only for a $\gamma$-approximate maximizer of $\dist(\cdot,V_j)$, whereas $\gamma=1$ requires the supremum to be attained; see Remark~\ref{rem:gamma1}.
The $\sigma_j$ are compared with the Kolmogorov widths $d_n(\mathcal F)_{\mathcal H}$ of \eqref{eq:dn-general} by the following two estimates from \cite{DPW13}.

\begin{proposition}[{\cite[Cor.~3.3(i) and eq.~(3.7)]{DPW13}}]\label{prop:dpw}
Let $\mathcal H$ be a Hilbert space, $\mathcal F\subset B_{\mathcal H}$, $\gamma\in(0,1]$, and $(\sigma_j)_{j\ge0}$ be as in \eqref{eq:weak-greedy}.
Then, writing $d_n\coloneqq d_n(\mathcal F)_{\mathcal H}$, it holds for all $n,s\in\N$ that
\begin{align}
    \sigma_{2n} &\le \sqrt2 \gamma^{-1} \sqrt{d_n}, \label{eq:DPWi}\\
    \sigma_{4s} &\le \sqrt2 \gamma^{-1} \sqrt{\sigma_{2s} d_s}. \label{eq:DPW37}
\end{align}
\end{proposition}

Although \cite{DPW13} assumes compactness, its proof only uses a minimizing subspace for $d_m$.
For bounded $\mathcal F$, choose instead a subspace with error at most $d_m+\varepsilon$, apply the same argument, and let $\varepsilon\to0$.
This proves the stated form.
Below, we prepare our main proofs.

If $R=0$, then $K=0$ and $\mathcal H_K=\{0\}$, all claims in Section \ref{sec:MainResults} are trivial; hence assume $R>0$.
We apply Proposition~\ref{prop:dpw} with $\mathcal H=\mathcal H_K$ and $\gamma=\tfrac12$, so that $\sqrt2 \gamma^{-1}=2\sqrt2$, and with the normalized dictionary
\begin{equation}\label{eq:Fdef}
    \mathcal{F}\coloneqq R^{-1}\mathcal{K}\subset B_{\mathcal H_K}.
\end{equation}
Note that $\sigma_0=1$ by the definition \eqref{eq:R} of $R$.
Every selected $f_j=a_{x_j}/R \in \mathcal F$ is a translate and hence an admissible sampling node.
Writing $P_j\coloneqq\{x_0,\dots,x_{j-1}\}$ for the selected nodes, so that $V_j=V_{P_j}$, we have $\dist(a_x/R,V_j)=\mathrm{Pow}_{P_j}(x)/R$ for every $x\in D$.
Hence $\sigma_j=\norm{\mathrm{Pow}_{P_j}}_\infty/R$, and the selection criterion \eqref{eq:weak-greedy} becomes
\begin{equation}\label{eq:pgreedy-rule}
    \mathrm{Pow}_{P_j}(x_j)\ \ge\ \tfrac12 \norm{\mathrm{Pow}_{P_j}}_\infty ,
\end{equation}
a relaxation of the \emph{$P$-greedy} point selection of the kernel literature \cite{SH17}, which takes a maximizer.
Finally, $d_n(\mathcal{F})_{\mathcal H_K}=d_n(\mathcal{K})_{\mathcal H_K}/R$ by homogeneity, so that Propositions~\ref{prop:gm} and~\ref{prop:cn} yield
\begin{equation}\label{eq:greedy-vs-g}
g_m^{\lin}(B_{\mathcal H_K})_\infty = \tau_m(\mathcal{K})_{\mathcal{H}_K} \le R \sigma_m,
\qquad
d_n(\mathcal{F})_{\mathcal H_K} = \frac{d_n(\mathcal{K})_{\mathcal H_K}}{R} = \frac{c_n(B_{\mathcal H_K})_\infty}{R}.
\end{equation}
In view of the first relation in \eqref{eq:greedy-vs-g}, it suffices to bound
$\sigma_m$.
By Lemma~\ref{lem:interp} and the proof of Proposition~\ref{prop:gm}, each resulting bound is realized by kernel interpolation at $P_m$, with worst-case error at most $\norm{\mathrm{Pow}_{P_m}}_\infty=R\sigma_m$.

\begin{remark}[Nested designs]\label{rem:nested}
The greedy construction yields nested designs $P_0\subset P_1\subset\cdots$ with at most one additional node per $P_m$.
Hence, all bounds hold along one sequence, the previous samples are reused, and the interpolant can be updated incrementally.
In contrast, guarantees based on drawing or subsampling a separate design for each budget may not have this property.
\end{remark}

\begin{remark}[The case $\gamma=1$]\label{rem:gamma1}
For $\gamma<1$, the selection only requires an approximate maximizer, whereas $\gamma=1$ is the classical $P$-greedy rule and requires the supremum of the power function to be attained.
This holds, for example, when $D$ is finite or when $D$ is compact and $K$ is continuous.
We use $\gamma=\tfrac12$ throughout; the general constants follow by replacing $2\sqrt2$ with $\sqrt2\gamma^{-1}$, with the best constants obtained for $\gamma=1$ whenever it is admissible.
\end{remark}
 
\subsection{Proofs of main results}
\begin{proof}[Proof of Theorem~\ref{thm:rates}]
Set $M\coloneqq\sup_{0\le k\le n}L(k) c_k/R$.
Dividing \eqref{eq:carl} by $R$ and using that $g_k^{\lin}\le R \sigma_k$ by \eqref{eq:greedy-vs-g}, it suffices to prove that
\begin{equation}\label{eq:carl-sd}
    \sigma_m \le C_L \frac{M}{L(m)}\qquad(0\le m\le n).
\end{equation}
Since $c_0=R$, we have $M\ge L(0) c_0/R=L(0)>0$, so that
\begin{equation}\label{eq:Qdef}
    Q\coloneqq\max_{0\le m\le n}\frac{\sigma_m L(m)}{M}
\end{equation}
is a well-defined maximum over finitely many terms, and \eqref{eq:carl-sd} is equivalent to $Q\le C_L$.
Let $m_\ast$ be an index at which the maximum in \eqref{eq:Qdef} is attained.
If $Q=0$, there is nothing to prove, so we assume $Q>0$ and distinguish two cases according to the size of $m_\ast$.
 
For $m_\ast\le3$, we get by monotonicity that $\sigma_{m_\ast}\le\sigma_0=1$.
Together with $M\ge L(0)$, this gives
\begin{equation}\label{eq:carl-small}
    Q=\frac{\sigma_{m_\ast}L(m_\ast)}{M} \le \frac{L(m_\ast)}{L(0)} \le \frac{L(3)}{L(0)}.
\end{equation}
Otherwise, we set $s\coloneqq\lfloor m_\ast/4\rfloor$, so that $s\ge1$, $4s\le m_\ast\le 4s+3\le 8s$ and $2s\le m_\ast/2\le n$.
By monotonicity, the self-improving recursion \eqref{eq:DPW37}, and the second relation in \eqref{eq:greedy-vs-g}, we obtain
\begin{equation}\label{eq:carl-rec}
    \sigma_{m_\ast}\le\sigma_{4s}\le 2\sqrt2 \sqrt{\sigma_{2s} c_s/R}.
\end{equation}
Further, the definition of $Q$ gives $\sigma_{2s}\le Q M/L(2s)$, and the definition of $M$ gives $c_s/R\le M/L(s)$.
Inserting these into \eqref{eq:carl-rec} and using that $L(m_\ast)\le L(8s)$, we obtain
\begin{equation}\label{eq:carl-large}
    Q=\frac{\sigma_{m_\ast}L(m_\ast)}{M} \le 2\sqrt2 \sqrt Q \frac{L(8s)}{\sqrt{L(2s)L(s)}} .
\end{equation}
Dividing \eqref{eq:carl-large} by $\sqrt Q>0$ and squaring gives $Q\le 8 L(8s)^2/(L(2s)L(s)) \leq 8\Lambda_L$ by \eqref{eq:LambdaL}.
The two cases give $Q\le L(3)/L(0)$ and $Q\le 8\Lambda_L$, respectively, and the claim follows.
\end{proof}
 
\begin{proof}[Proof of Theorem~\ref{thm:geom}]
By \eqref{eq:DPWi} and the second relation in \eqref{eq:greedy-vs-g}, we have $\sigma_{2n}\le  \sqrt{8c_n/R}$.
Multiplying by $R$ and using the first relation in \eqref{eq:greedy-vs-g} then gives
\begin{equation}
    g^{\lin}_{2n}\le R \sigma_{2n}\le R\sqrt{8c_n/R}= \sqrt{8R c_n}.
\end{equation}
This concludes the proof.
\end{proof}

\begin{proof}[Proof of Theorem~\ref{thm:log}]
Put $\delta\coloneqq c_n/R$.
We iterate the self-improving recursion \eqref{eq:DPW37} until $\sigma_{2^{k+1}n}\le 16\delta$, and then count the points used.
If $8\delta\ge1$, there is nothing to prove:
by \eqref{eq:greedy-vs-g} and monotonicity, we have $g_m^{\lin}\le R\sigma_m\le R\sigma_0=R\le16R\delta=16c_n$ for all $m \in \N$.
We may therefore assume $8\delta<1$.

By \eqref{eq:DPW37} with $s=2^kn\ge n$, the monotonicity $c_{2^kn}\le c_n$, and the second relation in \eqref{eq:greedy-vs-g}, we get
\begin{equation}\label{eq:log-rec}
\sigma_{2^{k+2}n}\le \sqrt{8\sigma_{2^{k+1}n}\delta},\qquad k\ge0,
\end{equation}
with the initial value $\sigma_{2n}\le\sigma_0=1$.
Dividing \eqref{eq:log-rec} by $8 \delta$ turns it into $\sigma_{2^{k+2}n}/(8\delta)\le\sqrt{\sigma_{2^{k+1}n}/(8\delta)}$, so that iteration and $\sigma_{2n}\le1$ give
\begin{equation}\label{eq:log-iter}
    \sigma_{2^{k+1}n}\le 8\delta \Big(\frac{1}{8\delta}\Big)^{1/2^k}, \qquad k\ge0 .
\end{equation}
Since $8\delta<1$, the factor $(8\delta)^{-1/2^k}$ is at most $2$ precisely when $2^{-k}\log_2\frac{1}{8\delta}\le1$, that is, when $k\ge\log_2\log_2\frac{1}{8\delta}$.
We therefore choose $k\coloneqq\max\bigl\{0,\lceil\log_2\log_2\tfrac{1}{8\delta}\rceil\bigr\}$, for which \eqref{eq:log-iter} gives $\sigma_{2^{k+1}n}\le16\delta$.
By the definition of $k$, we have
\begin{equation}\label{eq:EstimateSamples}
    2^{k+1}\le4\max\{1,\log_2 \tfrac{1}{8\delta}\}\le4\log_2\bigl(2+\tfrac{1}{\delta}\bigr)=4\log_2\bigl(2+\tfrac{R}{c_n}\bigr).
\end{equation}
We can multiply \eqref{eq:EstimateSamples} by $n$.
Then, using the monotonicity of sampling numbers and \eqref{eq:greedy-vs-g}, we conclude  for $m = \lceil 4n\log_2(2+R/c_n) \rceil$ that $g_m^{\lin}\le g_{2^{k+1}n}^{\lin}\le R\sigma_{2^{k+1}n}\le16R\delta=16c_n$.
\end{proof}

\begin{proof}[Proof of Theorem \ref{thm:basis_sel}]
Set $E\coloneqq\overline{\spann\mathcal A} \subset X$, and consider $J\colon E^*\to B(\mathcal A)$ given by $(J\ell)(a)=\ell(a)$, which is injective by the definition of $E$.
Note that $J\ell$ is indeed a bounded function by compactness of $\mathcal A$.
View $\mathcal{F}\coloneqq J(B_{E^*})$ as the unit ball of $E^*$ with the usual (dual) norm.
For fixed nodes $a_1,\ldots,a_k\in\mathcal A$ and coefficient functions $\varphi_i\in B(\mathcal A)$, the Hahn--Banach theorem gives
\begin{equation}
    \sup_{\ell\in B_{E^*}} \biggl\|J\ell-\sum_{i=1}^k\ell(a_i)\varphi_i\biggr\|_\infty
    =\sup_{a\in\mathcal A} \biggl\|a-\sum_{i=1}^k\varphi_i(a)a_i\biggr\|_X \geq \sup_{a\in\mathcal A}\dist_X \bigl(a,\spann\{a_1,\ldots,a_k\}\bigr),
\end{equation}
and hence $\tau_m(\mathcal A)_X\le g_m^{\lin}(\mathcal F)_\infty$.

Now fix a subspace $Y\subset X$ of dimension $r\le m$ with basis $y_1,\ldots,y_r \in X$, and let $\varepsilon>0$.
For each $a\in\mathcal A$, choose $y(a)=\sum_{j=1}^r\eta_j(a)y_j\in Y$ with $\eta_j(a) \in \C$, $j=1,\ldots,r$, such that $\|a-y(a)\|_X\le\dist_X(a,Y)+\varepsilon$.
Next, we study the properties of the $\eta_j$.
Let \(T\colon \mathbb C^r\to Y\) with \(T\mathbf c=\sum_{j=1}^r c_jy_j\) be the associated basis isomorphism.
Since \(\|y(a)\|_X\le2\sup_{b\in\mathcal A}\|b\|_X+\varepsilon\) and \(T^{-1}\) is bounded, we get \(\eta_j \in B(\mathcal A)\).
For each $\ell\in B_{E^*}$, let \smash{$\widetilde\ell\in B_{X^*}$} denote a Hahn--Banach extension.
Then, we have
\begin{align}
\sup_{\ell\in B_{E^*}}
\dist_{B(\mathcal A)}\bigl(J\ell,\spann\{\eta_1,\ldots,\eta_r\}\bigr)
&\le
\sup_{\ell\in B_{E^*}}
\biggl\|J\ell-\sum_{j=1}^r\widetilde\ell(y_j)\eta_j\biggr\|_\infty
\notag =
\sup_{\ell\in B_{E^*}}\sup_{a\in\mathcal A}
\bigl|\widetilde\ell\bigl(a-y(a)\bigr)\bigr|
\notag\\
&\le
\sup_{a\in\mathcal A}\|a-y(a)\|_X \le
\sup_{a\in\mathcal A}\dist_X(a,Y)+\varepsilon.
\end{align}
Taking the infimum over $Y$ with $\dim(Y) \leq m$ and letting $\varepsilon \to 0$ gives $d_m(\mathcal F)_{B(\mathcal A)}\le d_m(\mathcal A)_X$.
The first inequality in \eqref{eq:KPUU-sqrtn} now gives
\begin{equation}\label{eq:ProofFinal}
    \tau_{2m}(\mathcal A)_X
    \le g_{2m}^{\lin}(\mathcal F)_\infty
    \le C\sqrt m\,d_m(\mathcal F)_{B(\mathcal A)}
    \le C\sqrt m\,d_m(\mathcal A)_X,
\end{equation}
which concludes the proof.
\end{proof}
The factor $\sqrt m$ in \eqref{eq:ProofFinal} cannot be improved, not even under constant oversampling, as the following example shows.
Fix \(c>1\) and an integer \(L>c\).
By the construction in the proof of \cite[Prop.~5.13]{FR13}, for every sufficiently large prime \(m\) there is an explicit \(m\times m^2\) matrix with complex entries of coherence \(m^{-1/2}\).
In fact, it is observed that fixing the translation of the Alltop vector and varying the modulation yields an orthonormal basis of $\C^m$.
Different translations yield different bases.
Hence, we get $m$ different orthonormal bases of $\C^m$.
Select \(L\leq m\) such bases and denote their synthesis matrices by \(\mathbf U_1,\ldots,\mathbf U_L\in\mathbb C^{m\times m}\).
The coherence bound implies
\begin{equation}
    \left|(\mathbf U_r^*\mathbf U_\ell)_{ij}\right|=m^{-1/2}, \qquad r\neq \ell. 
\end{equation}
We consider the space $X_m$ of $m\times L$ matrices equipped with the $\ell_\infty$-norm of their entries and the dictionary $\mathcal{A}_m$ of canonical unit matrices, namely 
\begin{equation}
    X_m=\ell_\infty^{m\times L}, \qquad \mathcal A_m=\{\mathbf e_{j,\ell}\in \mathbb{C}^{m\times L}:\ 1\leq j\leq m,1\leq\ell\leq L\}, 
\end{equation}
where the matrices $\mathbf e_{j,\ell} \in \mathbb{C}^{m\times L}$ contain just one non-vanishing entry $1$ at the position $(j,\ell)$.

Since \(\lfloor cm\rfloor<Lm\), every selection of \(\lfloor cm\rfloor\) elements of \(\mathcal A_m\) omits a coordinate vector.
Hence, it holds $\tau_{\lfloor cm\rfloor}(\mathcal A_m)_{X_m}=1$.
Consider now the \(m\)-dimensional subspace of $\C^{m\times L}$ given by 
\begin{equation}
    Y_m\coloneqq\left\{ \mathbf U(\mathbf x)\coloneqq (\mathbf U_1^*\mathbf x| \dots |\mathbf U_L^*\mathbf x) \in \C^{m\times L} : \mathbf x\in\mathbb C^m \right\}.
\end{equation}
For any \(\mathbf e_{j,\ell}\in \mathcal{A}_m\) choose \(\mathbf x=\mathbf U_\ell\mathbf e_j\), where $\mathbf e_j$ is the $j$-th canonical unit vector in $\C^m$ and also the $\ell$-th column of $\mathbf e_{j,\ell}$.
Then the \(\ell\)-th column in $\mathbf U(\mathbf x)$ is exactly \(\mathbf e_j\), while every remaining column has entries with  modulus not larger than \(m^{-1/2}\).
Therefore
\begin{equation}
    \|\mathbf e_{j,\ell}-\mathbf U(\mathbf U_{\ell}\mathbf e_j)\|_{X_m} \leq m^{-1/2}
\end{equation}
and hence $d_m(\mathcal A_m)_{X_m}\leq m^{-1/2}$.
Consequently, we obtain
\begin{equation}
    \frac{\tau_{\lfloor cm\rfloor}(\mathcal A_m)_{X_m}} {d_m(\mathcal A_m)_{X_m}} = \frac{1} {d_m(\mathcal A_m)_{X_m}} \geq \sqrt{m}.
\end{equation}
Thus, the factor \(\sqrt{m}\) cannot generally be replaced by \(o(\sqrt{m})\), even under constant oversampling.
For the precise index in Theorem \ref{thm:basis_sel}, take \(c=2\) and \(L=3\).
The example is a sequence of finite compact sets rather than one fixed compact set; that suffices to establish sharpness of a universal estimate.

\section{Examples and numerical illustration}\label{sec:numerics}
For the unit ball of a real RKHS $\mathcal H_K$ on a domain $D\subset\R^d$, we want to compare the linear sampling widths $g^{\lin}_m(B_{\mathcal H_K})_\infty=\inf_{|P|\le m}\|\mathrm{Pow}_P\|_\infty$ against the Gelfand widths.
By definition, we have $c_m(B_{\mathcal H_K})_\infty\le g^{\lin}_m(B_{\mathcal H_K})_\infty$, and Theorem~\ref{thm:rates} bounds the two against each other.
As a numerical consistency check, we compute the recovery error on three kernel families: the Legendre kernel $K_s$ with $s=2,3$ on $[-1,1]$, the periodic mixed-Sobolev kernel $H^m_{\mathrm{mix}}([0,1]^d)$ for $(d,m)\in\{2,3\}\times\{1,2,3\}$ except $(3,3)$, and the band-limited Paley--Wiener kernel $K_c$ with $2c/\pi\in\{20,40,80\}$ on $[-1,1]$.

As an estimate of $g_m^{\lin}$, we use
$\widehat g_m^{\lin}\coloneqq\max_{x\in X_{\rm sel}}\mathrm{Pow}_{P_m}(x)$
for the $P$-greedy set $P_m$ and the candidate grid $X_{\rm sel}$.
Algorithm~\ref{alg:pgreedy} summarizes the implementation together with the resulting kernel interpolant.
Its selection rule is with $\gamma=1$ relative to $X_{\rm sel}$.
Thus, the greedy criterion holds on all of $D$ with $\gamma=\tfrac12$ only if the grid resolves $\sup_{x \in D}\mathrm{Pow}_{P_j}(x)$ within a factor $2$, which we cannot verify numerically.
Two opposite biases therefore act on the computed  errors. Taking the greedy set rather than the optimal one raises them, by \eqref{eq:greedy-vs-g}, while evaluating its supremum on a grid lowers them.
Instead of solving the generally ill-conditioned Gram system, we use the Newton basis produced by the $P$-greedy recursion \cite{MS09,PS11}.
This basis is the $\mathcal H_K$-orthonormalization of the selected translates, which, equivalently, is a pivoted Cholesky factorization.
So the coefficients $\beta_j=\ip{f}{v_j}_{\mathcal H_K}$ are obtained from the samples by the forward substitution in Algorithm~\ref{alg:pgreedy} and satisfy $\sum_j|\beta_j|^2\le\norm{f}_{\mathcal H_K}^2$. Although no matrix inversion is required, the approach becomes unstable as the pivot values $\mathrm{Pow}_{P_j}(x_j)$ approach zero.
Stabilized variants are discussed in \cite{WSH21}.
The grids are Chebyshev-type for the Legendre kernel, and equidistant ($d=1$) or scrambled Sobol ($d>1$) for the other examples.
Their sizes are specified for each example individually.
Further implementation details and the code used to produce the figures are available on GitHub\footnote{\url{https://github.com/Zeppo1994/Gelfand-Numbers-and-P-Greedy.git}}.

For the Gelfand widths, we compute lower estimates as follows.
Let $\rho$ be any probability measure on $D$ and let $C_\rho=\int_D a_x\otimes a_x\,\mathrm d\rho(x)$ have eigenvalues $\lambda_0\geq\lambda_1\geq\cdots\geq0$.
For every subspace $V$ of dimension at most $n$, the variational principle then gives
\begin{equation}\label{eq:SpectralTailBound}
\sup_{x\in D}\operatorname{dist}_{\mathcal H_K}(a_x,V)^2
 \geq \int_D\|(I-P_V)a_x\|_{\mathcal H_K}^2 \mathrm d\rho(x)
 =\operatorname{tr}\bigl((I-P_V)C_\rho\bigr)
 \geq\sum_{k=n}^\infty\lambda_k.
\end{equation}
Taking the infimum over $V$ therefore yields the spectral-tail lower bound $c_n(B_{\mathcal H_K})_\infty^2\geq\sum_{k=n}^\infty\lambda_k$, which has been observed in the literature multiple times, see \cite{CoKuSi16},\cite{OP95}.

\begin{algorithm}[tbp]
\small
\captionsetup{skip=2pt}
\caption{$P$-greedy selection and kernel interpolation}\label{alg:pgreedy}
\begin{algorithmic}[1]
\Require kernel $K$, grid $X_{\rm sel}$, node budget $m$, squared-power tolerance $\varepsilon_{\rm pow}>0$, for interpolation: point evaluations of $f$
\State $P_0\gets\varnothing$; $\mathrm{Pow}_{P_0}(x)^2\gets K(x,x)$ for $x\in X_{\rm sel}$; $\widehat m\gets m$
\For{$j=0,\ldots,m-1$}
  \State $x_j\gets\argmax_{x\in X_{\rm sel}}\mathrm{Pow}_{P_j}(x)$
  \If{$\mathrm{Pow}_{P_j}(x_j)^2\le\varepsilon_{\rm pow}$} $\widehat m\gets j$; \textbf{break}  \EndIf\vspace{-.75ex} 
  \State $v_j(x)\gets(K(x,x_j)-\smash[b]{\sum_{\ell<j}}v_\ell(x)v_\ell(x_j))/\mathrm{Pow}_{P_j}(x_j)$ for $x\in D$ \Comment{Newton basis}
  \State $P_{j+1}\gets P_j\cup\{x_j\}$; $\mathrm{Pow}_{P_{j+1}}(x)^2\gets \max\{\mathrm{Pow}_{P_j}(x)^2-\vert v_j(x)\vert^2, 0 \}$ \Comment{power update}
\EndFor\vspace{-.75ex}
\State $\widehat g_{\widehat m}^{\lin}\gets\max_{x\in X_{\rm sel}}\mathrm{Pow}_{P_{\widehat m}}(x)$ \Comment{reported sampling estimate}
\For{$j=0,\ldots,\widehat m-1$} \Comment{optional: stable kernel interpolation}
  \State $\beta_j\gets\bigl(f(x_j)-\smash[b]{\sum_{\ell<j}}\beta_\ell v_\ell(x_j)\bigr)/\mathrm{Pow}_{P_j}(x_j)$ \Comment{forward substitution}
\EndFor\vspace{-.75ex}
\State \Return $P_{\widehat m}$, $\widehat g_{\widehat m}^{\lin}$, and optionally $\Pi_{V_{P_{\widehat m}}}f=\sum_{j<\widehat m}\beta_j v_j$
\end{algorithmic}
\end{algorithm}

\subsection{Examples}\label{sec:Numerical_results}

\begin{figure}[t]
  \centering
  \includegraphics[width=0.9\linewidth]{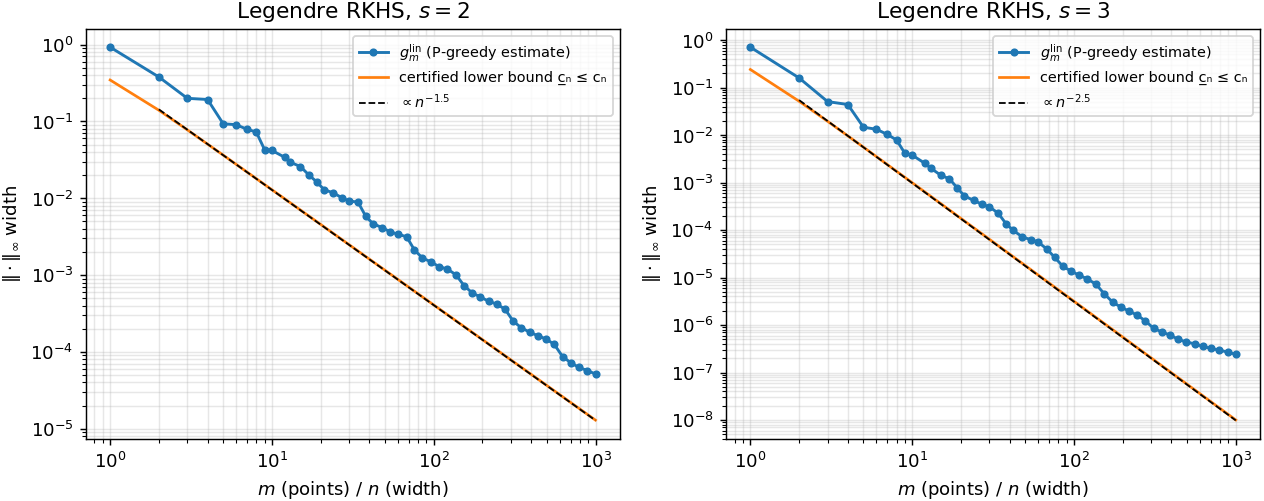}
  \caption{Legendre kernel $K_s$, $s=2, 3$.
  The $P$-greedy estimate of $g_m^{\lin}$ and the
  computed lower Gelfand width estimate $\underline c_n$ decay at the common rate
  $n^{-(s-1/2)}$ (dashed).}
  \label{fig:legendre}

  \vspace{.5cm}
  \includegraphics[width=\linewidth]{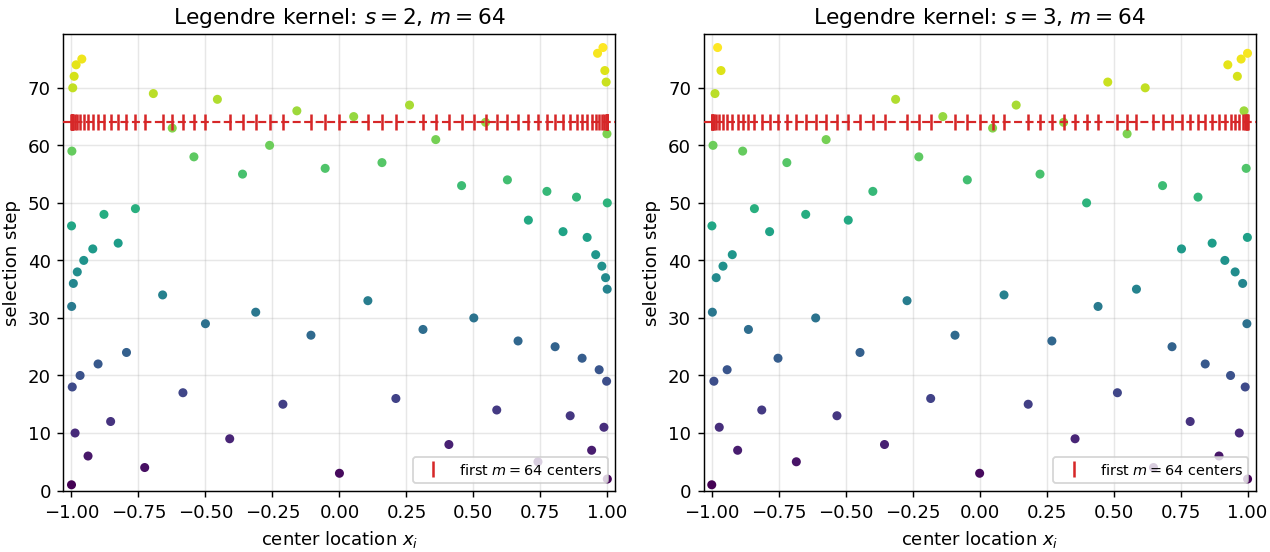}
  \caption{P-greedy design for Legendre kernel $K_s$, $s=2, 3$: selection order $m$ versus node position
  $x\in[-1,1]$.
  The red line at $m=64$ shows points selected up to that threshold.}
  \label{fig:legendre-points}
\end{figure}
\paragraph{Legendre kernel.}
On $D=[-1,1]$ with the Lebesgue measure, let $P_k$ denote the $L^2([-1,1])$-normalized Legendre polynomials, i.e.\ $P_k=\sqrt{(2k+1)/2}\,L_k$.
Following \cite[\S7.3]{PU}, we take the Mercer kernel
\begin{equation}\label{eq:legendre}
  K_s(x,y)=\sum_{k\ge0}\mu_k\,P_k(x)\,P_k(y),
  \qquad
  \mu_k=\frac{1}{1+\bigl(k(k+1)\bigr)^{s}},
\end{equation}
with eigenpairs $(\mu_k,P_k)$.
The eigenvalues decay as $\Theta(k^{-2s})$, giving a bounded kernel for $s>1$.
For the uniform probability measure $\mathrm d\rho=\mathrm dx/2$, the eigenvalues of $C_\rho$ are $\mu_k/2$.
Thus, the spectral-tail estimate \eqref{eq:SpectralTailBound} gives
\begin{equation}\label{eq:legendre-tail-lower}
 c_n(B_{\mathcal H_{K_s}})_\infty^2
 \geq \underline c_{n,s}^2
 \coloneqq
 \frac12\sum_{k=n}^{\infty}
 \frac{1}{1+[k(k+1)]^s},
 \qquad
 \underline c_{n,s}\asymp
 \frac{n^{-(s-1/2)}}{\sqrt{2(2s-1)}}.
\end{equation}
The \(P\)-greedy computations use \(20\,000\) Chebyshev candidate points.
In the following, we consider $s=2,3$.
Figure~\ref{fig:legendre} shows the $P$-greedy estimate of $g_m^{\lin}$ and the computed lower width estimate $\underline c_n$.
For $s=3$ and $m$ large, we hit the numerical precision limit.
The found P-greedy design clusters towards the endpoints, see Figure \ref{fig:legendre-points}.

\paragraph{Periodic mixed-Sobolev kernel.}
\begin{figure}[t]
  \centering
  \includegraphics[width=\linewidth]{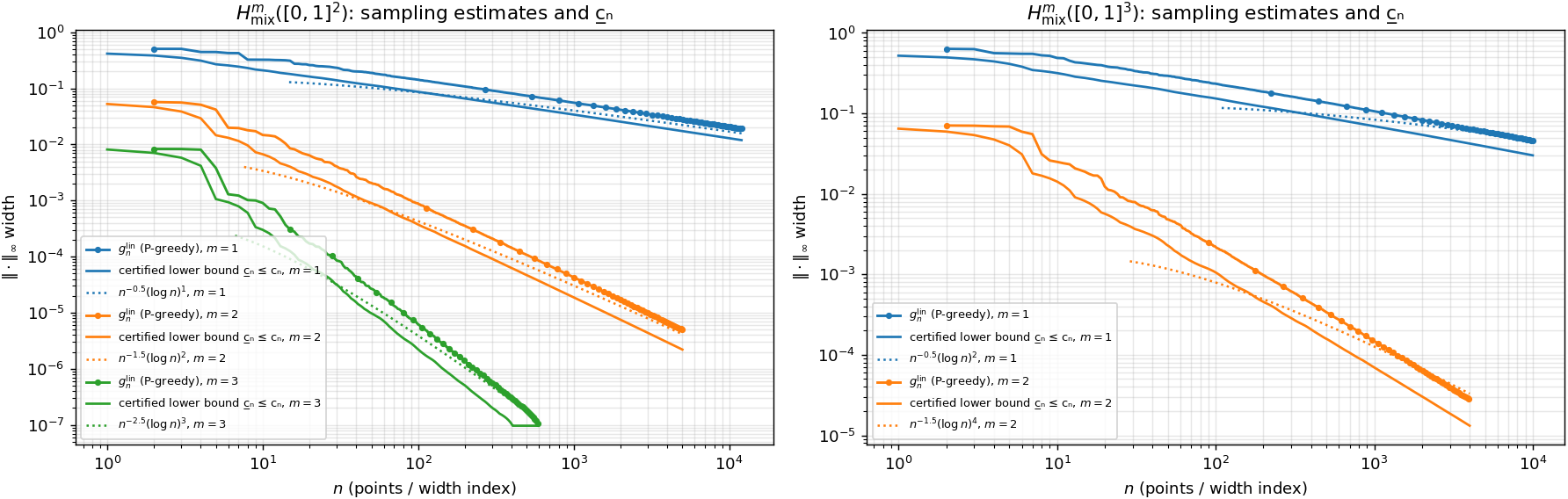}
  \caption{Periodic mixed-Sobolev kernel $H^m_{\mathrm{mix}}([0,1]^d)$: $P$-greedy estimates of $g_n^{\lin}$ versus the
  computed lower Gelfand width estimate for $(d,m)\in\{2,3\}\times\{1,2,3\}$ except $(3,3)$.
  The mixed-smoothness rate $n^{-(m-1/2)}(\log n)^{(d-1)m}$ is overlaid as an asymptotic guide.}
  \label{fig:periodic-rates}
  \vspace{.5cm}

  \centering
  \includegraphics[width=\linewidth]{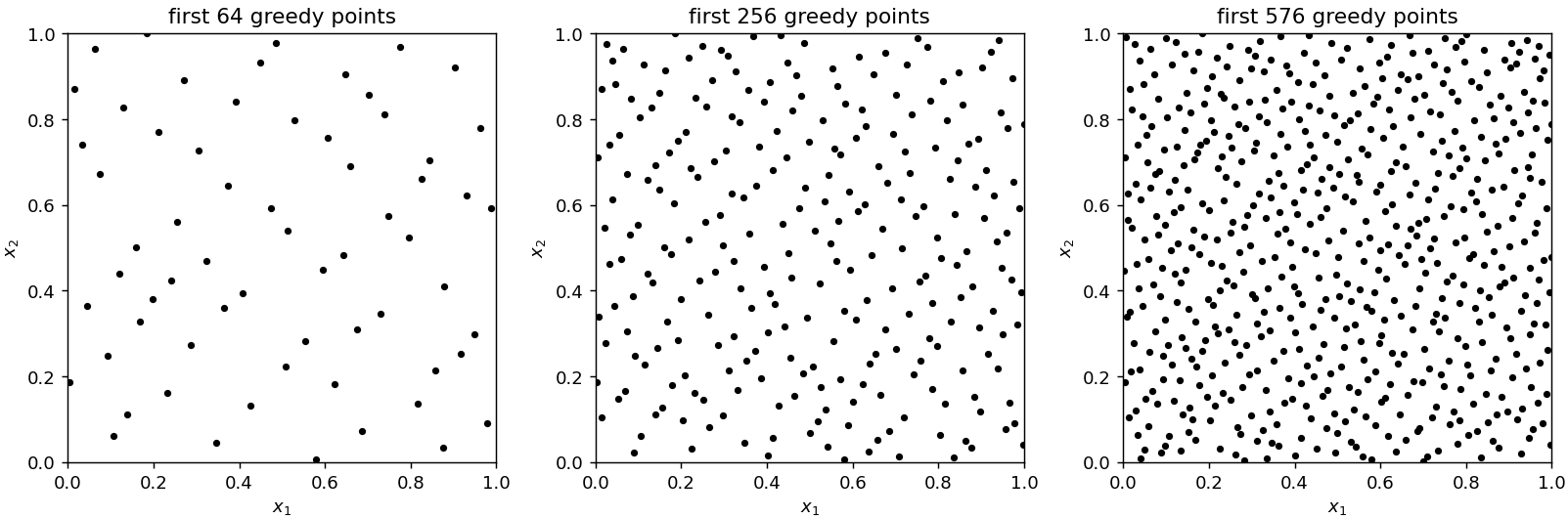}
  \caption{$P$-greedy designs in $[0,1]^2$ for $H^m_{\mathrm{mix}}$ with $64,256,576$ selected points.}
  \label{fig:periodic-points}
\end{figure}
On the torus $D=\mathbb{T}^d = [0,1]^d$, we take the reproducing kernel of the mixed-smoothness Sobolev space $H^m_{\mathrm{mix}}(\mathbb{T}^d)$, cf.\ \cite[Chapt.\ 3]{DuTeUl16}, which is the $d$-fold tensor product of the univariate kernel (\cite[p.\ 318]{BerlinetThomasAgnan})
\begin{equation}\label{eq:mixed}
  k_1(x,y)=1+\frac{(-1)^{m-1}}{(2m)!} B_{2m} \bigl(\{x-y\}\bigr),
  \qquad
  K(\mathbf x,\mathbf y)=\prod_{j=1}^{d}k_1(x_j,y_j),
\end{equation}
where $B_{2m}$ is the Bernoulli polynomial of degree $2m$, $m\in\mathbb N$, and $\{\cdot\}$ is the fractional part.
Its Fourier diagonalization is $k_1(x,y)=\sum_{k\in\mathbb Z}\mu_k\,e^{2\pi i k(x-y)}$ with
$\mu_0=1$ and $\mu_k=(2\pi|k|)^{-2m}$ for $k\neq0$.
We consider the nonincreasing rearrangement of the tensor sequence $(\mu_{k_1}\cdots\mu_{k_d})_{\mathbf{k}\in \Z^d}$, which is denoted by $(\lambda_k)_{k\in \N_0}$.
The asymptotic and preasymptotic behavior of such sequences has been studied intensively, see \cite{KuSiUl15}. Since the Fourier modes have constant modulus, the subspace spanned by the first $n$ modes has the same squared pointwise residual $\sum_{j=n}^\infty\lambda_j$ at every point.
Hence, the spectral-tail lower bound \eqref{eq:SpectralTailBound} is attained for odd $n$ and
\begin{equation}\label{eq:periodic-tail}
 c_n(B_{H^m_{\mathrm{mix}}(\mathbb T^d)})_\infty^2
 =\sum_{j=n}^{\infty}\lambda_j \asymp n^{-2m+1}(\log n)^{2m(d-1)}.
\end{equation}
This agrees with \cite[Thm.~3.4]{CoKuSi16}.
We provide numerical  results for $(d,m)\in\{2,3\}\times\{1,2,3\}$ except $(3,3)$.
It is already known that $c_n \asymp g_n^{\lin}$, where the $g_n^{\lin}$ have been realized by Smolyak sparse grids \cite{Bungartz_Griebel_2004}, \cite{DuTeUl16} which turned out to be optimal \cite{Te93}.
Recently, it has been shown in \cite{KPUU23} that subsampled random points also serve as an optimal design.
For $d=2$ and $d=3$, the $P$-greedy grid contains $120\,000$ and $80\,000$ points, respectively.
Figures~\ref{fig:periodic-rates} and~\ref{fig:periodic-points} show this very behavior also with $P$-greedy points.
Across the illustrated cases, the ratios of the $P$-greedy estimates to $\underline c_n$ remain bounded, even though at reachable $n$ the logarithmic factor of the rate is far from developed.
Thus, the direct ratio, and not the asymptotic rate line, is the correct comparison.
Here, the selected greedy set fills $[0,1]^2$ quasi-uniformly, see Figure~\ref{fig:periodic-points} for three configurations, reflecting the stationarity of the kernel and the absence of boundary concentration.

\paragraph{Band-limited (Paley--Wiener) kernel.}
\begin{figure}[t]
  \centering
  \includegraphics[width=\linewidth]{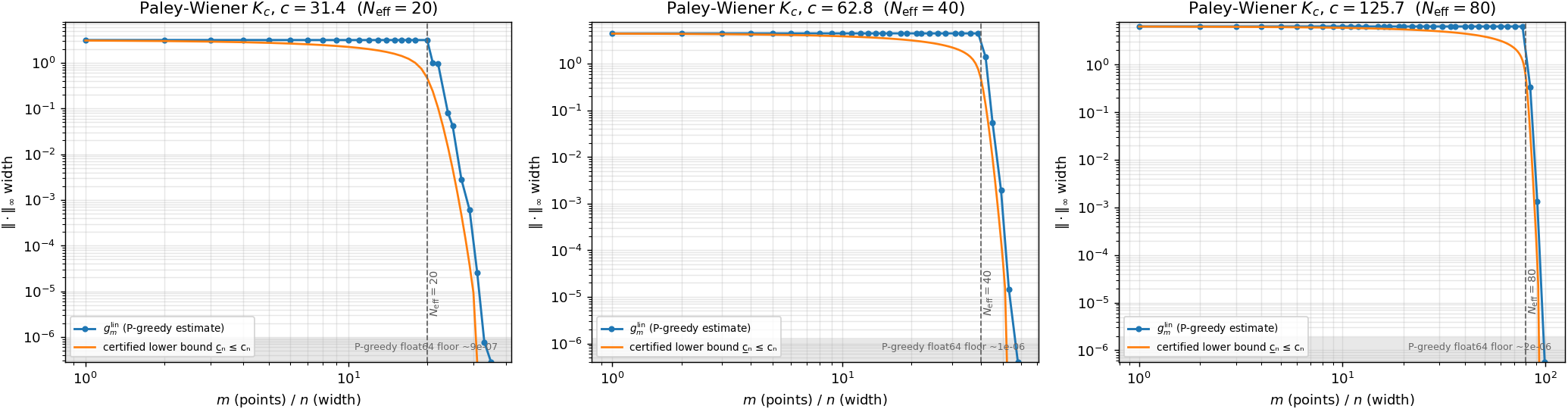}
  \caption{Paley--Wiener kernel $K_c$, $N_{\mathrm{eff}}\in\{20,40,80\}$.
  The $P$-greedy estimate of $g_m^{\lin}$ and the lower Gelfand width estimate are all $\approx$ flat up to $N_{\mathrm{eff}}$ (dashed
  line) and then decay super-exponentially into the float64 floor (shaded).}
  \label{fig:paley-widths}

  \vspace{.5cm}
  \includegraphics[width=\linewidth]{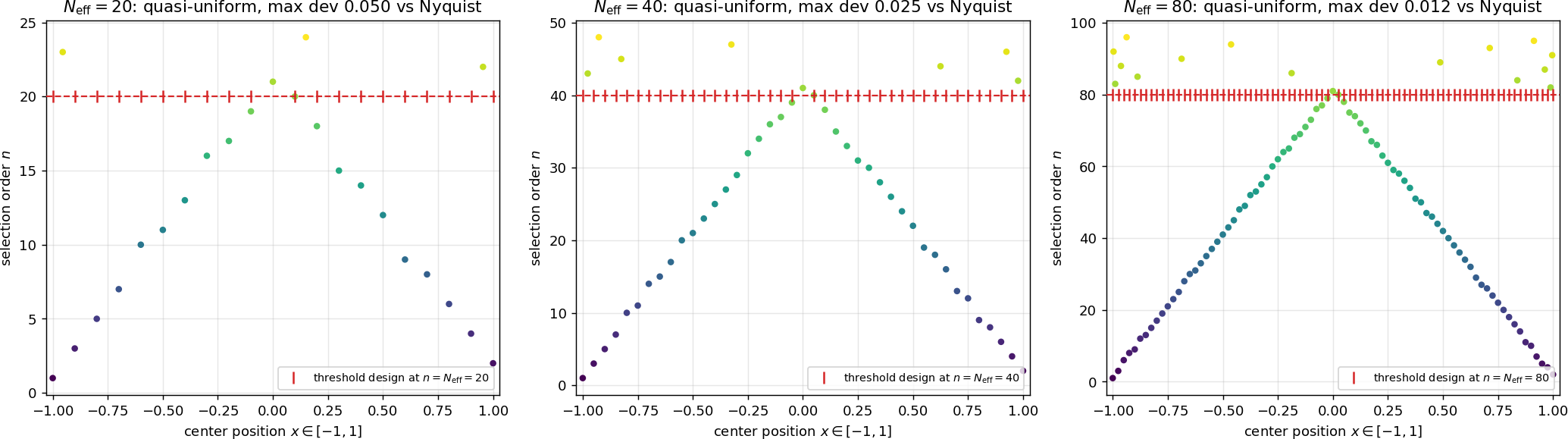}
  \caption{P-greedy design for Paley--Wiener kernel $K_c$,
  $N_{\mathrm{eff}}\in\{20,40,80\}$: selection order $m$ versus node position
  $x\in[-1,1]$.
  The red line at $m=N_{\mathrm{eff}}$ shows points selected up to that threshold.}
  \label{fig:paley-points}
\end{figure}
On $D=[-1,1]$, we consider the Paley--Wiener space of band-limited functions with bandwidths $[-c,c]$.
Its reproducing kernel reads as
\begin{equation}\label{eq:paleywiener}
  K_c(s,t)=\frac{\sin \bigl(c(s-t)\bigr)}{\pi(s-t)},
  \qquad K_c(x,x)=\frac{c}{\pi}.
\end{equation}
The singular values for the embedding into $L_2([-1,1])$ are the
square roots $\sigma_k=\sqrt{\lambda_k(c)}$ of the Slepian--Pollak
eigenvalues \cite{SlPo61,BonamiKaroui17} of the time-and-band-limiting
integral operator on $[-1,1]$. For $\mathrm d\rho=\mathrm dx/2$, the
spectral-tail estimate gives
\begin{equation}\label{eq:paley-tail-lower}
 c_n(B_{\mathcal H_{K_c}})_\infty^2
 \geq \underline c_{n,c}^2
 \coloneqq
 \frac12\sum_{k=n}^{\infty}\lambda_k(c)
 =
 \frac12\left(
 N_{\mathrm{eff}}-\sum_{k=0}^{n-1}\lambda_k(c)
 \right),
 \qquad
 N_{\mathrm{eff}}=\frac{2c}{\pi}.
\end{equation}
For the numerical experiment, we instead use a numerically tractable lower bound originating from the Gauss--Legendre probability measure $\rho_q=\sum_{\ell=1}^q\omega_\ell\delta_{x_\ell}$ on $[-1,1]$.
Let \smash{$\vartheta_0^{(q)}\geq\cdots\geq \vartheta_{q-1}^{(q)}\geq0$} be the eigenvalues of the covariance
matrix
\begin{equation}
    \mathbf A_q=\bigl(\sqrt{\omega_\ell}
K_c(x_\ell,x_r)\sqrt{\omega_r}\bigr)_{\ell,r=1}^q.
\end{equation}
Since $\|f\|_\infty\geq\|f\|_{L_2(\rho_q)}$, every such discrete
probability measure gives
\begin{equation}\label{eq:paley-discrete-lower}
c_n(B_{\mathcal H_{K_c}})_\infty^2
\geq
\bigl(\underline c_{n,c}^{(q)}\bigr)^2
\coloneqq
\sum_{k=n}^{q-1}\vartheta_k^{(q)},
\qquad 0\leq n<q.
\end{equation}
For fixed $n$, this discrete tail converges to the spectral tail in \eqref{eq:paley-tail-lower} as the quadrature is refined.
We use $q=210,370,690$ for $N_{\mathrm{eff}}=20,40,80$, respectively.
For $n\gtrsim N_{\mathrm{eff}}$, Theorem~\ref{thm:rates} does not reproduce this cliff.
In fact, for any weight function $L(k) \coloneqq (1+k)^{\alpha}$ admitted by Remark~\ref{rem:RV}, the supremum $\sup_{k\le n}L(k)c_k$ is governed by the plateau near $N_{\mathrm{eff}}$, 
yielding only
\begin{equation}
    g_n^{\lin}\lesssim 
C_L R \frac{L(N_{\mathrm{eff}})}{L(n)}\lesssim_\alpha \sqrt{c}\Big(\frac{c}{n}\Big)^{\alpha} 
\end{equation}
rather than super-exponential decay. 
Thus, we rather compare with Theorem~\ref{thm:geom}. 

We use three bandwidths, $N_{\mathrm{eff}}\in\{20,40,80\}$, whose cliffs move to the right in proportion to $c$. Because the kernel is band-limited, the Gram matrix $\mathbf G$ with $N\gg N_{\mathrm{eff}}$ nodes has numerical rank
$\approx N_{\mathrm{eff}}$.
For $N_{\mathrm{eff}}=20,40,80$, respectively, the $P$-greedy grids contain $8\,000$, $8\,000$, and $10\,000$ points.
Figure~\ref{fig:paley-widths} shows the numerical results.
The $P$-greedy estimates of $g_m^{\lin}$ and the lower Gelfand width estimate stay approximately flat up to $N_{\mathrm{eff}}$ and then decay rapidly.
Across the cliff, the sampling estimates and width estimate differ by up to about one decade over the few indices before they reach $\tau$.

We do not claim that this experiment illustrates Theorem~\ref{thm:geom}. After the cliff its bound $\sqrt{8R c_n}$ is several orders of magnitude weaker than the observed values, while the pre-existing (non-constructive) bound \eqref{eq:KPUU-sqrtn} is stronger there.
The experiment shows only that greedy sampling remains competitive with the computed Gelfand width estimates well past the range in which our sufficient conditions yield reasonable bounds.
After $n\approx N_{\mathrm{eff}}$, all computed values  rapidly approach the double-precision floor (shaded).
Interestingly, the greedy approach places points quasi-uniformly in $[-1,1]$ at the Nyquist spacing $2/N_{\mathrm{eff}}$, see Figure~\ref{fig:paley-points} for three configurations.
 
\paragraph{Acknowledgment.}
TU and KP are supported by the German Research Foundation (DFG) with grant Ul403/4-1.
Several of the presented relations and the code were developed with the assistance of LLMs by Anthropic and OpenAI.
Additional verification, presentation and writing were carefully carried out by the authors.

\bibliographystyle{abbrv}

\end{document}